\documentclass[12pt,a4paper,reqno,oneside]{amsart}
\usepackage[includeheadfoot,margin=3.5cm,headsep=1.5cm,footskip=0.5cm]{geometry}

\usepackage{amsmath,amsthm,amssymb}
\usepackage{mathtools}
\usepackage{mathptmx}
\usepackage{array}
\usepackage[colorlinks=true,linkcolor=blue,citecolor=blue]{hyperref}

\usepackage[compact]{titlesec}
\titleformat{\section}{\large\bfseries}{\thesection.}{0.3em}{}
\titlespacing*{\section}{0pt}{5pt}{10pt}
\titleformat{\subsection}{\normalfont\bfseries}{\thesubsection.}{0.3em}{}
\titlespacing*{\subsection}{1em}{5pt}{10pt}
\usepackage{indentfirst}
\newtheorem{theorem}{Theorem}[section]
\newtheorem{lemma}{Lemma}[section]
\newtheorem{proposition}{Proposition}[section]
\newtheorem{corollary}{Corollary}[section]
\newtheorem{definition}{Definition}[section]
\theoremstyle{remark}
\newtheorem{remark}{Remark}[section]
\newtheorem{example}{Example}[section]

\begin{document}
\title{Nonlinear contraction in \(b\)-suprametric spaces} 
\author{Maher Berzig}
\address{\noindent Maher Berzig, 
Universit\'e de Tunis,  \'Ecole Nationale Sup\'erieure d'Ing\'enieurs de Tunis,   D\'epartement de Math\'ematiques,   5 avenue Taha Hussein Montfleury, 1008 Tunis, Tunisie.}
\email{maher.berzig@gmail.com}

\begin{abstract}
We introduce  the concept of \(b\)-suprametric spaces and establish a fixed point result for mappings satisfying  a nonlinear contraction in such spaces.  The obtained result  generalizes a fixed point theorem of Czerwik  and a recent result of the author. 
\end{abstract}

	
\maketitle

\section{Introduction}
\label{sec1}
\indent The triangle inequality is one of the most important and useful inequalities in analysis.
However, the full power of the triangle inequality  is not necessarily essential for certain applications, such as in the measures used for the distance between shapes \cite{Fagin1998}, the problems of  traveling salesman  or in the nearest neighbor search  \cite{Gragera2018}, and even for the study of denotational semantics of data  in program verification~\cite{Matthews1992}.
There are currently several generalizations of the triangle inequality satisfied by various distance functions, such as those of the partial metric \cite{Matthews1992},  the $b$-metric \cite{Czerwik1993}  or  the rectangular metric \cite{Branciari2000}. \\
\indent  The partial metric was introduced by Matthews  \cite{Matthews1992} by modifying the triangular inequality  to handle  non-zero self-distance, 
which  allows him to  extend the  Banach contraction  principle  and to develop some applications in programming language theory. 
Further investigations  have since been carried out to explore the usefulness of this distance function, see for instance \cite{Bugajewski2020,Gregori2019,Hoc2022,Kumar2021}.\\
\indent A relaxed triangular inequality, via a real coefficient $b$ greater or equal to one, was used by Czerwik  \cite{Czerwik1993} and the space equipped with  a distance satisfying such an inequality called it $b$-metric space. Furthermore, he established a fixed point theorem in the context of $b$-metric spaces.
Due to the importance of this distance function, it has been extended  in various ways and adapted to a wide range of applications, see e.g.   \cite{Afshari2020,Berzig2018,Boriceanu2010,Lazreg2021}.\\
\indent  In \cite{Branciari2000}, Branciari replaced the triangle inequality of the  distance function  with a rectangular one, which is also called  a quadrilateral inequality.  
The rectangular distance between two points is less than or equal to the sum of their  distances to any two distinct points and the distance between them.
Currently, there are many fixed point results developed in spaces equipped with distance functions satisfying   rectangular type inequalities, we refer to the non-exhaustive list of references \cite{Baradol2020,Budhia2020,Turinici2021,Patil2022}. \\
\indent In addition, and in order to cope with the increasing complexity of real-world applications, many new distance functions appear by merging, relaxing or extending some of the axioms of  the existing distance functions.
Moreover, several  fixed point results have been developed in sets equipped with these distances,  see for  instance the results obtained in the context of  the partial rectangular metric spaces \cite{Shukla2014}, $b$-rectangular metric spaces \cite{Roshan2016},  extended-Branciari $b$-distance spaces \cite{Abdeljawad2020}, quasi-partial $b$-metric space \cite{Gautam2021} and  extended $b$-metric spaces \cite{Younis2022b}. 
\par Recently, a new relaxed triangular inequality  is used in the so-called suprametric, which was  introduced  by the author  in \cite{Berzig2022}.  In particular, the Banach contraction principle is proven in  suprametric spaces, along with other results in ordered vector spaces.  Some  nonlinear integral and matrix equations have  been studied  in  this context.
The suprametric was used  by Panda et al. \cite{Panda2023a} to analyze  the  stability and  the existence of equilibrium points of some fractional-order complex-valued neural networks. Additionally, the suprametric space was extended by  Panda et al.  \cite{Panda2023b} and a study of the  existence of a solution of Ito-Doob type stochastic integral equations was proven via this extension. However,  it is well known that the space of integrable functions $L_{1}[0,1]$  is a $b$-metric space (see for instance \cite{Berinde1993}), but it is not known whether the $b$-metric  is a suprametric or  an extended suprametric. In order to study the existence of positive solutions to certain nonlinear equations in this space, it will be useful  to  develop a context including the suprametric and the $b$-metric spaces. \\
\indent In this paper, we explore an important step toward this task by unifying  these two concepts and  introduce the $b$-suprametric spaces, which merge the advantages  of both the $b$-metric and the suprametric spaces.
Then, we establish a fixed point theorem for mappings satisfying nonlinear contractions via Matkowski functions.
\section{Preliminaries}
\label{sec2}
Throughout this paper, we respectively denote by $\mathbb{R}$, $\mathbb{R}_{+}$ and $\mathbb{N}$ the sets of  all   real numbers,  all  nonnegative real numbers  and all nonnegative natural numbers. 
\begin{definition}\label{semimetric}
Let $X$ be a nonempty set. A function   $d\colon X\times X \to \mathbb{R}_{+}$ is called semimetric if   for all $x,y\in X$ the following properties hold:
\begin{enumerate}
\item[\rm(d\textsubscript{1})]\label{d1} $d(x,y)=0$ if and only if $x=y$,
\item[\rm(d\textsubscript{2})]\label{d2} $d(x,y)=d(y,x)$.
\end{enumerate}
A semimetric space is a pair $(X,d)$ where $X$ is a  nonempty set and $d$ is a semimetric. 
\end{definition}

Recall now the concept of $b$-metric spaces \cite{Czerwik1993}.
\begin{definition}\label{bmetric}
Let $(X,d)$ be a semimetric space  and $b\ge 1$. The function   $d$ is called $b$-metric if  it satisfies the following additional property:
\begin{enumerate}
\item[\rm(d\textsubscript{3})]\label{b3} $d(x,y)\le b\,(d(x,z)+d(z,y))$ for all $x,y,z\in X$.
\end{enumerate}
A $b$-metric space is a pair $(X,d)$ where $X$ is a  nonempty set and $d$ is a $b$-metric. 
\end{definition}
For more information on the topology of the $b$-metric spaces, we refer the reader to   \cite{VanAn:2015}.  Let us recall a few  examples of such  spaces.
\begin{example}\label{Example1}
Let $(X,d)=(\ell_{p}(\mathbb{R}),d_{\ell_{p}})$ with $0<p<1$ where 
\begin{equation*}
\ell_{p}(\mathbb{R})\coloneqq\big \{\{x_{n}\}\subset \mathbb{R}\text{ such that } \textstyle\sum\limits_{n=1}^{\infty}|x_{n}|^{p}<\infty\big \},
\end{equation*}
 and $d_{\ell_{p}}\colon X\times X \to \mathbb{R}_{+}$ is given by
\begin{equation*}
d_{\ell_{p}}(\{x_{n}\},\{y_{n}\})=\left(\textstyle\sum\limits_{n=1}^{\infty}|x_{n}-y_{n}|^{p}\right)^{\textstyle\frac{1}{p}} \text{ for all }\{x_{n}\},\{y_{n}\}\in X.
\end{equation*}
Then $(X,d_{\ell_{p}})$ is a $b$-metric space with $b=2^{\frac{1}{p}}$.
\end{example}
\begin{example}\label{Example2}
Let $(X,d)=(L_{p}[0,1],d_{L_{p}})$ with $0<p<1$ where 
\begin{equation*}
L_{p}[0,1]\coloneqq\left\{x\colon[0,1]\to \mathbb{R}\text{ such that } \int_{0}^{1}|x(t)|^{p}<\infty\right\},
\end{equation*}
 and $d_{L_{p}}\colon X\times X \to \mathbb{R}_{+}$ is given by
\begin{equation*}
d_{L_{p}}(x,y)=\left (\int_{0}^{1}|x(t)-y(t)|^{p}\right )^{\textstyle\frac{1}{p}}\text{ for all }x,y\in X.
\end{equation*}
Then $(X,d_{L_{p}})$ is a $b$-metric space with $b=2^{\frac{1}{p}}$.
\end{example}
In the sequel, we denote  by $\mathbb{M}$ the set of Matkowski  functions \cite{Matkowski1975}, $\psi\colon \mathbb{R}_{+}\to \mathbb{R}_{+}$ that satisfy 
\begin{enumerate}
\item[\rm(i)] $\psi$ is increasing on $\mathbb{R}_{+}$, 
\item[\rm(ii)] $ \lim\limits_{n\to \infty}\psi^{n}(t)=0$ for all $t>0$.
\end{enumerate}
Denote   by $\mathbb{M}_{b}$ $(b\ge1)$ the set of   functions of $\mathbb{M}$  that satisfy
\begin{equation}\label{psi3}
\limsup_{n\to\infty}\textstyle\frac{\psi^{n+1}(t)}{\psi^{n}(t)}<\frac{1}{b}, \;\text{ for all }t>0,
\end{equation}
\begin{remark}\label{rem}
It is not difficult to see that if $\psi\in\mathbb{M}$, then $\psi(t)<t$ for all  $t>0$, $\psi$ is continuous at $0$ and $\psi(0)=0$. 
\end{remark}
Recall next the  result of Czerwik  \cite{Czerwik1993}, see also Kaj{\'a}nt{\'o} and Luk{\'a}cs \cite{Kajanto2018}. 
\begin{theorem}\label{Czerwik}
Let $(X,d)$ be a complete \(b\)-metric space and $f\colon X\to X$ be a mapping. Assume   there exists $\psi\in\mathbb{M}$ such that $f$ is a $\psi$-contraction, that is,
\begin{equation*} 
d(fx,fy)\le \psi(d(x,y))\;\text{ for all }x,y\in X.
\end{equation*}
Then, $f$ has a unique fixed point  and  $\{f^{n}x_{0}\}_{n\in\mathbb{N}}$ converges to it for every $x_{0}\in X$.
\end{theorem}
The suprametric space introduced in \cite{Berzig2022} as follows.
\begin{definition}\label{b-suprametric}
Let $(X,d)$ be a semimetric space  and $\rho\in \mathbb{R}_{+}$. The function $d$ is called suprametric if  it satisfies the following additional property:
\begin{enumerate}
\item[\rm(d\textsubscript{3})]\label{s3} $d(x,y)\le d(x,z)+d(z,y)+\rho\, d(x,z)d(z,y)$  for all $x,y,z\in X$.
\end{enumerate}
A suprametric space is a pair $(X,d)$ where $X$ is a  nonempty set and $d$ is a suprametric. 
\end{definition}
\begin{example}\label{example3}
Let  $a>0$, $b\ge1$, $d_{m}$ be a metric on $X$   and $d^{(a,b)}\colon X\times X \to \mathbb{R}_{+}$ be a function given by
\begin{equation*}
d^{(a,b)}(x,y)=d_{m}(x,y)(a\,d_{m}(x,y)+b)  \text{ for all }x,y\in X.
\end{equation*}
Then $(X,d^{(a,b)})$ is a suprametric  with  $\rho=\frac{2a}{b}$. 
\end{example}

\begin{example}\label{example4}
Let $\beta\ge1$, $d_{m}$ be  a metric on $X$   and $d^{(\beta)}_{1}\colon X\times X \to \mathbb{R}_{+}$ be a function  given by  
\begin{equation*}
d^{(\beta)}_{1}(x,y)= (e^{-\beta d_{m}(x,y)^{2}}-1) \text{ for all }x,y\in X.
\end{equation*}
Then $(X,d^{(\beta)}_{1})$ is a suprametric  with  $\rho=1$. 
\end{example}
\begin{example}\label{example5}
Let $\gamma>0$, $d_{m}$ be a metric on $X$   and $d^{(\gamma)}_{2}\colon X\times X \to \mathbb{R}_{+}$ be  a function given by  
\begin{equation*}
d^{(\gamma)}_{2}(x,y)=\gamma\, (e^{d_{m}(x,y)}-1) \text{ for all }x,y\in X.
\end{equation*}
Then $(X,d^{(\gamma)}_{2})$ is a suprametric  with  $\rho=\frac{1}{\gamma}$. 
\end{example}
\begin{remark}
Let $X=\mathbb{R}$ and  $d_{m}(x,y)=|x-y|$ for all $x,y\in X$. If $d$ is equal to $d^{(1,1)}$, $d^{(1)}_{1}$ or $d^{(1)}_{2}$, we get $d(0,1)+d(1,2)<d(0,2)$, which proves that there are not classical metrics.
\end{remark}
Recall now a fixed point result from \cite{Berzig2022}.
\begin{theorem}\label{Berzig}
Let $(X,d)$ be a complete suprametric space and $f\colon X\to X$ be a mapping. Assume   there exists $c\in[0,1)$ such that 
\begin{equation*} 
d(fx,fy)\le c\,d(x,y) \;\text{ for all }x,y\in X.
\end{equation*}
Then, $f$ has a unique fixed point  and $\{f^{n}x_{0}\}_{n\in\mathbb{N}}$ converges to it for every $x_{0}\in X$.
\end{theorem}
We next introduce the concept of $b$-suprametric spaces.
\begin{definition}\label{b-supra}
Let $(X,d)$ be a semimetric space, $b\ge1$ and $\rho\ge0$. The function  $d$ is called $b$-suprametric if  it satisfies the following additional property:
\begin{enumerate}
\item[\rm(d\textsubscript{3})]\label{d3} $d(x,y)\le b\,(d(x,z)+d(z,y))+\rho d(x,z)d(z,y)$  for all $x,y,z\in X$.
\end{enumerate}
A $b$-suprametric space is a pair $(X,d)$ where $X$ is a  nonempty set and $d$ is a $b$-suprametric. 
\end{definition}
\begin{example}\label{example6}
Let  $d\colon X\times X \to \mathbb{R}_{+}$ be given by
\begin{equation*}
d(x,y)=d^{(1,2)}(x,y)(d^{(1,2)}(x,y)+1)  \text{ for all }x,y\in X.
\end{equation*}
Then $(X,d)$ is a $b$-suprametric  with $(b,\rho)=(1,1)$.
\end{example}
\begin{example}\label{example7}
Let $d\colon X\times X \to \mathbb{R}_{+}$ be given by  
\begin{equation*}
d(x,y)= e^{-d^{(1)}_{1}(x,y)^{2}}-1\text{ for all }x,y\in X.  
\end{equation*}
Then $(X,d)$ is a $b$-suprametric  with  $(b,\rho)=(1,1)$.
\end{example}
\begin{example}\label{example8}
Let  $d\colon X\times X \to \mathbb{R}_{+}$ be given by
\begin{equation*}
d(x,y)=d_{\ell_{p}}(x,y)(d_{\ell_{p}}(x,y)+1)\text{ for all }x,y\in X.  
\end{equation*}
Then $(X,d)$ is a $b$-suprametric   with  $(b,\rho)=(4^{\frac{1}{p}},8^{\frac{1}{p}})$.
\end{example}
\begin{example}\label{example9}
Let $d\colon X\times X \to \mathbb{R}_{+}$ be given by  
\begin{equation*}
d(x,y)= d_{L_{p}}(x,y)(d_{L_{p}}(x,y)+1)\text{ for all }x,y\in X.  
\end{equation*}
Then $(X,d)$ is a $b$-suprametric   with  $(b,\rho)=(4^{\frac{1}{p}},8^{\frac{1}{p}})$.
\end{example}
Inspired from \cite[Example 3.9]{VanAn:2015}, we  consider a non-trivial structure and endow it with a $b$-suprametric.
\begin{example}\label{ex6}
Let $X=\{0,1,\frac{1}{2},\ldots,\frac{1}{n},\ldots\}$ and
\begin{equation*}
d(x,y)=
\left\{
\begin{array}{ll}
0 & \text{ if }x=y\\
\frac{1}{5} & \text{ if }x\ne y \text{ and }x,y\in \{0,1\}\\
1-e^{-|x-y|} & \text{ if }x\ne y \text{ and }x,y\in \{0\}\cup\{\frac{1}{2n},n=1,2,\ldots\}\\
\frac{1}{4}& \text{ otherwise}.
\end{array}
\right .
\end{equation*}
Then 
\begin{enumerate}
\item[\rm(i)] $d$ is a $b$-suprametric space on $X$ with $b=\frac{3}{2}$ and $\rho=7$.
\item[\rm(ii)] $d$ is not continuous in each variable.
\item[\rm(iii)] $B(1,\frac{1}{4})\coloneqq\big\{x\in X: d(1,x)<\frac{9}{40}\big\}$ is not open. 
\end{enumerate}
For the proof of the example, we need the following lemma. 
\begin{lemma}\label{lem}
We have:
\begin{enumerate}
\item[\rm(i)] $e^{-|s-t|}-1\le e^{-s}-e^{-t}$ for all $s,t>0$.
\item[\rm(ii)] $e^{-s}+e^{-(t+s)}\le e^{-t}+1$ for all $s,t>0$.
\item[\rm(iii)] $e^{-s}+e^{-t}\le e^{-|s-t|}+1$ for all $s,t>0$.
\end{enumerate}
In particular, for all $b\ge 1$ we have
\begin{enumerate}
\item[\rm(i')] $b(2-e^{-t}-e^{-|s-t|})\ge 1-e^{-s}$ for all $s,t>0$.
\item[\rm(ii')] $b(2-e^{-s}-e^{-(t+s)})\ge 1-e^{-t}$ for all $s,t>0$.
\item[\rm(iii')] $b(2-e^{-t}-e^{-s})\ge 1-e^{-|s-t|}$ for all $s,t>0$.
\end{enumerate}
\end{lemma}
\begin{proof}
\par \par\noindent{\rm (i)}:\; The case $s=t$ is trivial. Assume now that $t>s$, then by using that $t\mapsto \frac{1+e^{-t}}{1+e^{t}}$ is strictly decreasing for $t>0$, we obtain
$\frac{1+e^{-s}}{1+e^{s}}>\frac{1+e^{-t}}{1+e^{t}}=e^{-t}$,
thus $e^{-|s-t|}-1=e^{s-t}-1<e^{-s}-e^{-t}$.
Assume next that $s>t$, so  by using that $t\mapsto \frac{1-e^{-t}}{e^{t}-1}$ is strictly deceasing  for  $t>0$, we deduce that $\frac{1-e^{-t}}{e^{t}-1}>\frac{1-e^{-s}}{e^{s}-1}=e^{-s}$, so $e^{-|s-t|}-1=e^{t-s}-1<e^{-s}-e^{-t}$. 
\par\noindent{\rm (ii)}:\; Obvious.
\par\noindent{\rm (iii)}:\; The case $s=t$ is trivial.  Assume that $t> s$. Then from $-e^{-s}=\frac{e^{-s}-1}{e^{s}-1}<e^{-t}$, we obtain $e^{-|s-t|}+1=e^{s-t}+1>e^{-s}+e^{-t}$. Similarly, if we assume that $s> t$, we  deduce from $-e^{-t}=\frac{e^{-t}-1}{e^{t}-1}<e^{-s}$ that $ e^{-|s-t|}+1=e^{t-s}+1>e^{-s}+e^{-t}$. 
\end{proof}
\begin{proof}[Proof of Example \ref{ex6}]
\par\noindent{\rm (i)}:\; Note that for all $x,y\in X$, $d(x,y)\ge 0$, $d(x,y)=0$ if and only if $x=y$. We next discuss  several cases according to different values of $x,y,z (z\notin\{x,y\})$. Denote by $\delta(x,z,y)\coloneqq b(d(x,z)+d(z,y))+\rho\,d(x,z)d(z,y)$, where $b=\frac{3}{2}$ and $\rho=7$.
\par\noindent\textbullet~If $d(x,y)=d(0,1)=\frac{1}{5}$, then
$\delta(x,z,y)=
\frac{29}{8}-\frac{13}{4} e^{-\frac{1}{2n}}$
if $z=\frac{1}{2n}$ and $\frac{19}{16}$ otherwise.
\par\noindent\textbullet~If $d(x,y)=d(0,\frac{1}{2n})=1-e^{-\frac{1}{2n}}$, then
\begin{small}
\begin{equation*}
\delta(x,z,y)=
\left\{
\begin{array}{ll}
\frac{3}{2}(2-e^{-\frac{1}{2m}}-e^{-|\frac{1}{2m}-\frac{1}{2n}|})+4 (1-e^{-\frac{1}{2m}})(1-e^{-|\frac{1}{2m}-\frac{1}{2n}|}) & \text{ if } z=\frac{1}{2m}\\
\frac{41}{40} & \text{ if } z=1\\
\frac{19}{16} & \text{ otherwise}.
\end{array}
\right .
\end{equation*}
\end{small}
\par\noindent\textbullet~If $d(x,y)=d(0,\frac{1}{2n+1})=\frac{1}{4}$ $(n\ne 0)$, then
\begin{small}
\begin{equation*}
\delta(x,z,y)=
\left\{
\begin{array}{ll}
\frac{29}{8}-\frac{13}{4}e^{-\frac{1}{2m}}& \text{ if } z=\frac{1}{2m}\\
\frac{41}{40}  & \text{ if } z=1\\
\frac{19}{16} & \text{ otherwise}.
\end{array}
\right .
\end{equation*}
\end{small}
\par\noindent\textbullet~If $d(x,y)=d(\frac{1}{2n},1)=\frac{1}{4}$, then
\begin{small}
\begin{equation*}
\delta(x,z,y)=
\left\{
\begin{array}{ll}
\frac{29}{8}-\frac{13}{4} e^{-|\frac{1}{2n}-\frac{1}{2m}|}& \text{ if } z=\frac{1}{2m}\\
\frac{16}{5}-\frac{29}{10}e^{-\frac{1}{2n}} & \text{ if } z=0\\
\frac{19}{16} & \text{ otherwise}.
\end{array}
\right .
\end{equation*}
\end{small}
\par\noindent\textbullet~If $d(x,y)=d(\frac{1}{2n+1},1)=\frac{1}{4}$ $(n\ne 0)$, then
$\delta(x,z,y)=  \frac{41}{40}$ if $z=0$ and $\frac{19}{16}$  otherwise.
\par\noindent\textbullet~If $d(x,y)=d(\frac{1}{2m},\frac{1}{2n})=1-e^{-|\frac{1}{2m}-\frac{1}{2n}|}$, then
\begin{small}
\begin{equation*}
\delta(x,z,y)=
\left\{
\begin{array}{ll}
\frac{3}{2}(2-e^{-|\frac{1}{2m}-\frac{1}{2k}|}-e^{-|\frac{1}{2k}-\frac{1}{2n}|})+ 4 (1-e^{-|\frac{1}{2m}-\frac{1}{2k}|})(1-e^{-|\frac{1}{2k}-\frac{1}{2n}|}) & \text{ if } z=\frac{1}{2k}\\
\frac{3}{2}(2-e^{-\frac{1}{2m}}-e^{-\frac{1}{2n}}) + 4 (1-e^{-\frac{1}{2m}})(1-e^{-\frac{1}{2n}})& \text{ if } z=0\\
\frac{19}{16} & \text{ otherwise}.
\end{array}
\right .
\end{equation*}
\end{small}
\par\noindent\textbullet~If $d(x,y)=d(\frac{1}{2m},\frac{1}{2n+1})=\frac{1}{4}$ $(n\ne 0)$, then
\begin{small}
\begin{equation*}
\delta(x,z,y)=
\left\{
\begin{array}{ll}
\frac{29}{8}-\frac{13}{4}e^{-|\frac{1}{2m}-\frac{1}{2k}|}& \text{ if } z=\frac{1}{2k}\\
\frac{29}{8}-\frac{13}{4}e^{-\frac{1}{2m}} & \text{ if } z=0\\
\frac{19}{16} & \text{ otherwise}.
\end{array}
\right .
\end{equation*}
\end{small}
\par\noindent\textbullet~If $d(x,y)=d(\frac{1}{2m+1},\frac{1}{2n+1})=\frac{1}{4}$, then $\delta(x,z,y)=\frac{19}{16} $ for all $z\in X$.
\par\noindent We conclude by using Lemma \ref{lem} that
\begin{small}
\begin{equation*}
d(x,y)\le b (d(x,z)+d(z,y))+\rho d(x,z)+d(z,y),
\end{equation*}
where $b=\frac{3}{2}$ and $\rho=7$.
\end{small}
\par\noindent{\rm (ii)}:\; We have
\begin{equation*}
\lim_{n\to\infty}d(0,\textstyle\frac{1}{2n})=\lim\limits_{n\to\infty}(1-e^{-\frac{1}{2n}})=0.
\end{equation*}
However, we also have
\begin{equation*}
\lim_{n\to\infty}d(1,\textstyle\frac{1}{2n})=\textstyle\frac{1}{4}\ne \frac{1}{5}=d(1,0).
\end{equation*}
This proves that $d$ is not continuous in each variable.
\par\noindent{\rm (iii)}:\; Observe that  $B(1,\frac{9}{40})=\{0,1\}$. Since $d(0,\textstyle\frac{1}{2n})=1-e^{-\frac{1}{2n}}$, then for each $r>0$, $\frac{1}{2n}\in B(0,r)$ for a sufficiently large $n$ and since  $\frac{1}{2n}\not\in B(1,\frac{9}{40})$,  $B(0,r)\not\subset B(1,\frac{9}{40})$, which proves  that $B(1,\frac{9}{40})$ is not open.
\end{proof}
\end{example}
The following definition will be important  in what follows.
\begin{definition}
Let $(X,d)$ be a $b$-suprametric space.  
\begin{enumerate} 
\item[\rm(i)]  The sequence  $\{x_{n}\}_{n\in\mathbb{N}}$  converges to $x \in X$ if and only if $\lim\limits_{n\to\infty}d(x_{n},x)=0$.
\item[\rm(ii)] The sequence  $\{x_{n}\}_{n\in\mathbb{N}}$  is Cauchy if and only if $\lim\limits_{n,m\to\infty}d(x_{n},x_{m})=0$.
\item[\rm(iii)]  $(X, d)$ is complete if and only if any Cauchy sequence in $X$ is convergent.
\end{enumerate}
\end{definition}
\section{The main result}
\label{sec3}
The  main result is the following theorem.
\begin{theorem}\label{TH}
Let $(X,d)$ be a complete \(b\)-suprametric space and $f\colon X\to X$ be a mapping. Assume   there exists $\psi\in\mathbb{M}$ such that 
\begin{equation}\label{ctr}
d(fx,fy)\le \psi(d(x,y))\;\text{ for all }x,y\in X.
\end{equation}
Then, $f$ has a unique fixed point  and $\{f^{n}x_{0}\}_{n\in\mathbb{N}}$ converges to it for every $x_{0}\in X$.
\end{theorem}
\begin{proof}
Define $x_{n} = f^{n}x_{0}$ for all $n\in \mathbb{N}$, where $x_{0}$ is an arbitrary point of $X$.
For the sake of simplicity,  we will use the following  notations:
\begin{equation*}
d_{i,j}\coloneqq d(x_{i},x_{j})\text{ and } d_{i}\coloneqq d_{i,i+1}\;\text{ with }i,j\in\mathbb{N}.
\end{equation*}
We now apply \eqref{ctr}, $mn$ times, where $m,n\in\mathbb{N}$, we obtain
\begin{equation*}
d_{(m+1)n,mn}\le \psi^{mn}(d_{n,0}),
\end{equation*}
which implies 
\begin{equation}\label{lim}
\lim_{m\to\infty}d_{(m+1)n,mn}=0\,\text{ for all }n\in\mathbb{N}.
\end{equation}
Next, we shall show that $\{x_{n}\}$ is a Cauchy  sequence.  
Let $\varepsilon>0$ and take $q\in\mathbb{N}$ ($q>1$) such that
$\psi^{q}(\varepsilon)<\frac{\varepsilon}{b+\sqrt{b^2+\rho\varepsilon}}$.	 By using \eqref{lim},   one can choose  $p\in\mathbb{N}$ such that
\begin{equation}\label{key1}
d_{(m+1)q,mq}<\textstyle\frac{\varepsilon}{b+\sqrt{b^2+\rho\varepsilon}}\;\text{ for all }m\ge p.
\end{equation}
Hence, for all $z\in B(x_{pq},\varepsilon)\coloneqq\{x\in X: d(x_{pq},x)<\varepsilon\}$, we have
\begin{equation}\label{key2}
d(f^{q}z,f^{q}x_{pq})\le \psi^{q}(d(z,x_{pq}))\le  \psi^{q}(\varepsilon)<\textstyle\frac{\varepsilon}{b+\sqrt{b^2+\rho\varepsilon}},
\end{equation}
and since from \eqref{key1}, we have
\begin{equation}\label{key3}
d_{(p+1)q,pq}<\textstyle\frac{\varepsilon}{b+\sqrt{b^2+\rho\varepsilon}},
\end{equation}
then by Definition \ref{d3}, we obtain
\begin{align*}
d(f^{q}z,x_{pq}) 
&\le b \left (d(f^{q}z,f^{q}x_{pq})+d_{(p+1)q,pq}\right )+\rho d(f^{q}z,f^{q}x_{pq})\, d_{(p+1)q,pq}\\
&< 2b \textstyle\frac{\varepsilon}{{b+\sqrt{b^2+\rho\varepsilon}}^{\phantom{2}}}+\rho\textstyle\frac{\varepsilon^{2}}{ \left (b+\sqrt{b^2+\rho\varepsilon}\right )^{2}}=\varepsilon.
\end{align*}
We conclude that for all $\varepsilon>0$ there exist $p,q\in\mathbb{N}$ such that
\begin{equation}\label{key4}
f^{q}(B(x_{pq},\varepsilon))\subseteq B(x_{pq},\varepsilon).
\end{equation}
This implies  that for all $m\ge p$,   we have
\begin{equation}\label{key5}
d_{mq,pq}<\varepsilon.
\end{equation}
Next, we shall show that there exists $m_{0}\in\mathbb{N}$ such that
\begin{equation}\label{key6}
d_{mq,mq+k}<\varepsilon\;\text{ for all }m\ge m_{0}\text{ and all } k\in\{1,\ldots,q-1\}.
\end{equation}
Now, if $n=1$ in \eqref{lim}, we observe that there exists $i_{0}\in\mathbb{N}$ such that for all $i\ge i_{0}$, we have
\begin{equation*}
d_{i}<  \frac{\varepsilon}{\varepsilon+qc_{q}},
\end{equation*} 
where 
\begin{equation*}
c_{q}=\max\left\{(\begin{smallmatrix}q\\i\end{smallmatrix})b^{q-i}\rho^{i-1}:i=1,\ldots,q-1\right\}.
\end{equation*}
Let $m_{0}\in\mathbb{N}$ such that $m_{0}q>i_{0}$. We use Definition \ref{d3} and  deduce   that for all $m>m_{0}$ and all $k\in\{1,\ldots,q-1\}$,  we have
\begin{align*}
d_{mq,mq+k}
&\le \sum_{i=1}^{k}b^{k-i}\rho^{i-1} \textbf{e}_{i}(d_{mq},\ldots,d_{mq+k})\\
&\le  \sum_{i=1}^{k}  
\left(\begin{smallmatrix}k\\i\end{smallmatrix}\right)b^{k-i}\rho^{i-1}
\left ( \frac{\varepsilon}{\varepsilon+qc_{q}}\right )^{i}\\
&\le  \sum_{i=1}^{q-1}  
\left(\begin{smallmatrix}q\\i\end{smallmatrix}\right)b^{q-i}\rho^{i-1}
\left ( \frac{\varepsilon}{\varepsilon+qc_{q}}\right )^{i}\\
&\le 
\sum_{i=1}^{q-1} \left(\begin{smallmatrix}q\\i\end{smallmatrix}\right)b^{q-i}\rho^{i-1}
 \frac{\varepsilon}{\varepsilon+qc_{q}}\\
&\le 
 \sum_{i=1}^{q-1}  c_{q} \frac{\varepsilon}{qc_{q}}  =\varepsilon,
\end{align*}
where  $\textbf{e}$ is the elementary symmetric polynomial in $k$ variables, that is,
\begin{equation*}
\textbf{e}_{i}(x_{1},\ldots,x_{k})=\sum_{1\le j_{1}<j_{2}<\cdots<j_{i}\le k}\;\prod_{t=1}^{i}x_{j_{t}}
\quad\text{for  }i=1,\ldots,q-1.
\end{equation*}
Now, let $p_{0}=\max\{p,m_{0}\}$ and two integers $k_{1},k_{2}\ge p_{0}q$ and  $r_{1},r_{2}\in \{0,\ldots,q-1\}$ such that $k_{1}=p_{1}q+r_{1}$, $k_{2}=p_{2}q+r_{2}$ and $p_{1},p_{2}\ge p_{0}$. We conclude from \eqref{key5} and  \eqref{key6}, that we have
\begin{equation*}
\max\{u_{1},u_{2},u_{3},u_{4}\}<\varepsilon,
\end{equation*}
where 
\begin{gather*}
u_{1}\coloneqq d_{k_{1},p_{1}q},\quad
u_{2}\coloneqq d_{p_{1}q,pq},\quad
u_{3}\coloneqq d_{pq,p_{2}q},\quad 
u_{4}\coloneqq d_{p_{2}q,k_{2}}.
\end{gather*}
Hence, by \ref{d3} it follows
\begin{align*}
d_{k_{1},k_{2}}&\le  \left( u_{{3}}+u_{{4}} \right) {b}^{3}+u_{{2}}{b}^{2}+ \left( u_{{3}}u_{{4}}+u_{{2}} \left( u_{{3}}+u_{{4}} \right) +u_{{1}} \left( u_{{3}}
+u_{{4}} \right)  \right) \rho\,{b}^{2}\\
&\quad+ \left( u_{{2}}u_{{3}}u_{{4}}+
u_{{1}} \left( u_{{3}}u_{{4}}+u_{{2}} \left( u_{{3}}+u_{{4}} \right) 
 \right)  \right) b{\rho}^{2}+u_{{1}}u_{{2}}b\rho+bu_{{1}}+{\rho}^{3}u_{{1}}u_{{2}}u_{{3}}u_{{4}}\\
&\le  \varepsilon \left (\varepsilon^{3} \rho^{3}+4\varepsilon^{2} {b}\,\rho^{2}+5\varepsilon\rho\,{b}^{2}+\varepsilon{b}\,\rho+2 {b}^{3}+ b^{2}+ {b}\right )\\
&\le  \varepsilon \left (\varepsilon^{3} +\varepsilon^{2}+\varepsilon + 1\right ) \max\{\rho^{3},4\, {b}\,\rho^{2},6\,{b}^{2}\,\rho, 4\,{b}^{3}\}. 
\end{align*}
We conclude that $\{x_{n}\}$ is a Cauchy  sequence. As $(X,d)$ is complete, then   $\{x_{n}\}$  converges to some $x_{\ast}\in X$, say. Using  \ref{d3}, \eqref{ctr} and Remark \ref{rem}, we obtain
\begin{align*}
d(x_{\ast}, fx_{\ast})
&\le b\, d(x_{\ast}, x_{k+1})+b\, d(x_{k+1}, fx_{\ast})+\rho\,d(x_{\ast}, x_{k+1}) d(x_{k+1}, fx_{\ast})\\
&= b\, d(x_{\ast}, x_{k+1})+b\, d(fx_{k}, fx_{\ast})+\rho\,d(x_{\ast}, x_{k+1}) d(fx_{k}, fx_{\ast})\\
&\le b\, d(x_{\ast}, x_{k+1})+b\, \psi(d(x_{k}, x_{\ast}))+\rho\,d(x_{\ast}, x_{k+1}) \psi(d(x_{k}, x_{\ast})),
\end{align*}
thus as $k$ tends to infinity, we conclude that $x_{\ast}= fx_{\ast}$. 
Finally, assume that $x_{\ast}$ and $y_{\ast}$ are two distinct fixed points of $f$, then by \eqref{ctr} we deduce
\begin{equation*}
d(x_{\ast},y_{\ast})=d(fx_{\ast},fy_{\ast})\le \psi(d(x_{\ast},y_{\ast}))<d(x_{\ast},y_{\ast}),
\end{equation*}
which is a contradiction, so $x_{\ast}=y_{\ast}$ and the fixed point is unique.
\end{proof}
Next,  if $\psi\in\mathbb{M}_{b}$, we can use a proof similar to that of \cite[Theorem 2.1]{Berzig2022}.
\begin{theorem}\label{TH2}
Let $(X,d)$ be a complete \(b\)-suprametric space and $f\colon X\to X$ be a mapping. Assume   there exists $\psi\in\mathbb{M}_{b}$ such that \eqref{ctr} holds.
Then, $f$ has a unique fixed point  and $\{f^{n}x_{0}\}_{n\in\mathbb{N}}$ converges to it for every $x_{0}\in X$.
\end{theorem}
\begin{proof}
We adopt   the same notations  as the previous proof.
Assume that $d_{i}\ne 0$ for all $i$ (otherwise $x_{i}$ becomes a fixed point of $f$). So, from \eqref{ctr}, we deduce that $d_{i+1}\le \psi(d_{i})<d_{i}$  for all $i$. 
Next, we  show that $\{x_{i}\}$ is a Cauchy sequence. By Definition \ref{b-supra}, it follow that
\begin{align*}
d_{p,q}&\le bd_{p}+bd_{p+1,q}+\rho\,d_{p}\,d_{p+1,q}\\
&\le  b\psi^{p}(d_{0}) +(b+\rho\, \psi^{p}(d_{0}))  d_{p+1,q},
\end{align*}
where
\begin{align*}
d_{p+1,q} &\le bd_{p+1}+bd_{p+2,q}+\rho\,d_{p+1}\,d_{p+2,q}\\
&\le   b\psi^{p+1}(d_{0}) +(b+\rho\,  \psi^{p+1}(d_{0})) d_{p+2,q},
\end{align*}
Then combining the previous  inequalities, we obtain
\begin{align*}
d_{p,q}\le& 
b\psi^{p}(d_{0})  +b\psi^{p+1}(d_{0})(b+\rho\,  \psi^{p}(d_{0}))\\
&+(b+\rho\,  \psi^{p}(d_{0}))(b+\rho\,  \psi^{p+1}(d_{0}))d_{p+2,q}.
\end{align*}
Continuing this process, we obtain 
\begin{equation*}
d_{p,q}\le  \sum_{i=p}^{q-1} b\psi^{i}(d_{0})  \prod_{j=p}^{i-1} (b+\rho\psi^{j}(d_{0}) ).
\end{equation*}
We also have
\begin{equation*}
d_{p,q}\le  \sum_{i=p}^{q-1} b\psi^{i}(d_{0})  \prod_{j=0}^{i-1} (b+\rho\psi^{j}(d_{0}) ).
\end{equation*}
In order to study the behavior of $d_{p,q}$, we need to study the convergence of the series $\sum\limits_{i=0}^{\infty} u_{i}$, where 
\begin{equation*}
u_{i}=b\psi^{i}(d_{0})  \prod_{j=0}^{i-1} (b+\rho\psi^{j}(d_{0}) ).
\end{equation*}
To this end, observe that from  $\psi^{i+1}(d_{0}) \to 0$ as $i\to\infty$, we obtain 
\begin{equation*}
\lim_{i\to\infty}\frac{u_{i+1}}{u_{i}}\le \limsup_{i\to\infty}\frac{\psi^{i+1}(d_{0})}{\psi^{i}(d_{0})}(b+\rho\psi^{i}(d_{0}) )<1,
\end{equation*}
which implies that the series $\sum_{i=0}^{\infty} u_{i}$ converges, so $d_{p,q}$ tends to zero as $p,q$ tend to infinity.   Hence,   the sequence $\{x_{n}\}$ is  Cauchy, and  by completeness of $(X,d)$ we conclude that  $\{x_{n}\}$ converges to some $x\in X$.
We next show that $x$ is a fixed point of $f$. 
By using Definition \ref{b-supra} and \eqref{ctr}, we get 
\begin{align*}
d(x,fx)&\le b d(x,x_{n+1}) +b d(fx_{n},fx)+\rho  d(x,x_{n})d(fx_{n},fx)\\
&\le b d(x,x_{n+1}) +b \psi(d(x_{n},x)) + \rho d(x,x_{n}) \psi(d(x_{n},x))\\
&\le b d(x,x_{n+1}) +b d(x_{n},x) + \rho d(x,x_{n})^{2}.
\end{align*}
Thus, as $n$ tends to infinity, we deduce that $x=fx$. Finally, the uniqueness of the fixed point follows immediately from \eqref{ctr} and  Remark \ref{rem}.
\end{proof}
We immediately derive the following corollaries.
\begin{corollary}\label{COR1}
Let $(X,d)$ be a complete \(b\)-suprametric space and $f\colon X\to X$ be a mapping. Assume   there exists $c\in[0,1)$ such that 
\begin{equation*} 
d(fx,fy)\le c\,d(x,y)\;\text{ for all }x,y\in X.
\end{equation*}
Then, $f$ has a unique fixed point  and $\{f^{n}x_{0}\}_{n\in\mathbb{N}}$ converges to it for every $x_{0}\in X$.
\end{corollary}
\begin{corollary}\label{COR2}
Let $(X,d)$ be a complete suprametric space and $f\colon X\to X$ be a mapping. Assume there exists $\psi\in\mathbb{M}$ such that 
\begin{equation*} 
d(fx,fy)\le \psi(d(x,y))\;\text{ for all }x,y\in X.
\end{equation*}
Then, $f$ has a unique fixed point  and $\{f^{n}x_{0}\}_{n\in\mathbb{N}}$ converges to it for every $x_{0}\in X$.
\end{corollary}
As immediate consequences, we   obtain the following  propositions.
\begin{proposition}
Theorem \ref{TH} generalizes Theorem \ref{Czerwik}.  
\end{proposition}
\begin{proposition}
Corollary \ref{COR2} generalizes Theorem \ref{Berzig}.  
\end{proposition}

\noindent\textbf{Acknowledgment:} 
The author thanks the referees for valuable suggestions.

\section*{Compliance with Ethical Standards}
\noindent\textbf{Funding:} Not Applicable.\\
\textbf{Conflict of Interest:}  The author  declares that he  has no conflict of interest. \\
\textbf{Ethical approval:} This article does not contain any studies with human participants or animals performed by  the author.\\
\textbf{Data availability:} Not Applicable.

\end{document}